\documentclass[a4paper,fleqn]{cas-sc}

\usepackage[numbers]{natbib}

\usepackage{tikz}
\usepackage{subfig}
\usepackage{mathtools}
\usepackage{placeins}
\usepackage{pgfplots}
\DeclareFontFamily{U}{wncy}{}
\DeclareFontShape{U}{wncy}{m}{n}{<->wncyr10}{}
\DeclareSymbolFont{mcy}{U}{wncy}{m}{n}
\DeclareMathSymbol{\Sha}{\mathord}{mcy}{"58} 
\usepackage{bm}

\def\tsc#1{\csdef{#1}{\textsc{\lowercase{#1}}\xspace}}
\tsc{WGM}
\tsc{QE}

\newtheorem{theorem}{Theorem}
\newtheorem{lemma}[theorem]{Lemma}
\newdefinition{definition}{Definition}
\newdefinition{remark}{Remark}
\newproof{proof}{Proof}

\begin{document}
\let\WriteBookmarks\relax
\def\floatpagepagefraction{1}
\def\textpagefraction{.001}

\shorttitle{Constructing relaxation systems for lattice Boltzmann methods}    

\shortauthors{S. Simonis, M. Frank, M. J. Krause}  

\title[mode = title]{Constructing relaxation systems for lattice Boltzmann methods}



%

\author[1,2]{Stephan Simonis}[orcid=0000-0001-8555-4245]

\cormark[1]


\ead{stephan.simonis@kit.edu}


\credit{Conceptualization, 
Methodology, 
Software, 
Validation, 
Formal analysis, 
Investigation, 
Data curation, 
Writing - Original draft,  	
Visualization, 
Project administration}

\affiliation[1]{organization={Institute for Applied and Numerical Mathematics, Karlsruhe Institute of Technology},
            city={Karlsruhe},
            postcode={76131}, 
            country={Germany}}
\affiliation[2]{organization={Lattice Boltzmann Research Group, Karlsruhe Institute of Technology},
            city={Karlsruhe},
            postcode={76131}, 
            country={Germany}}

\author[1,3]{Martin Frank}




\credit{Conceptualization, 
Writing - Review {\&} Editing, 
Supervision}

\affiliation[3]{organization={Steinbuch Center for Computing, Karlsruhe Institute of Technology},
            city={Eggenstein-Leopoldshafen},
            postcode={76344}, 
            country={Germany}}

\author[1,2,4]{Mathias J. Krause}




\credit{Software, 
Resources, 
Writing - Review {\&} Editing, 
Supervision}

\affiliation[4]{organization={Institute of Mechanical Process Engineering and Mechanics, Karlsruhe Institute of Technology},
            city={Karlsruhe},
            postcode={76131}, 
            country={Germany}}

\cortext[1]{Corresponding author}



\begin{abstract}
We present the first top-down ansatz for constructing lattice Boltzmann methods (LBM) in \(d\) dimensions. 
In particular, we construct a relaxation system (RS) for a given scalar, linear, \(d\)-dimensional advection\textendash diffusion equation. Subsequently, the RS is linked to a \(d\)-dimensional discrete velocity Boltzmann model (DVBM) on the zeroth and first energy shell. 
Algebraic characterizations of the equilibrium, the moment space, and the collision operator are carried out. 
Further, a closed equation form of the RS expresses the added relaxation terms as prefactored higher order derivatives of the conserved quantity. 
Here, a generalized \((2d+1)\times (2d+1)\) RS is linked to a \(DdQ(2d+1)\) DVBM which, upon complete discretization, yields an LBM with second order accuracy in space and time. 
A rigorous convergence result for arbitrary scaling of the RS, the DVBM and conclusively also for the final LBM is proven. 
The top-down constructed LBM is numerically tested on multiple GPUs with smooth and non-smooth initial data in \(d=3\) dimensions for several grid-normalized non-dimensional numbers. 
\end{abstract}


\maketitle

\section{Introduction}

Lattice Boltzmann methods (LBM) have become a perfectly parallel alternative to conventional methods in computational fluid dynamics (CFD) and beyond \cite{lallemand2020lattice}. 
Several software realizations have been established, such as the open-source C++ framework OpenLB \cite{kummerlander2022olb15}. 
The parallel data structure enables multiphysics simulations with LBM on high-performance computing (HPC) machines \cite{krause2021openlb,
simonis2022forschungsnahe,
mink2021comprehensive,
dapelo2021lattice-boltzmann,
haussmann2021fluid-structure,
haussmann2019direct,
siodlaczek2021numerical, 
simonis2022temporal, 
bukreev2022consistent}. 
Further, OpenLB is suitable for studying the multi-dimensional stability sets of LBM itself \cite{simonis2021linear}. 

Nonetheless, the intrinsic relaxation principle of LBM stands in contrast to the direct design and analysis available for conventional top-down methods such as finite differences. 
As a consequence, the rigorous analysis of LBM is found incomplete \cite{simonis2020relaxation,bellotti2022finite}.  
As a first step towards a top-down derivation of LBM, we have proposed a constructive procedure for transforming a one-dimensional target PDE into a relaxation system (RS), which points to the typical moment system of LBM \cite{simonis2020relaxation}. 
This constructive ansatz for obtaining an LBM from a given target PDE is beneficial from various perspectives. 
First, the technique lifts the constraints of LBM in terms of guessing the moment system. 
With that, an LBM can be formulated for any PDE, which appears close enough to a balance or conservation law. 
Second, from the RS structure the correct limit towards solutions of the initial PDE can be ensured. 
Third, the added higher order derivatives, which are responsible for the bottom-up limiting property of LBM to the target PDE, are already exposed at the relaxation level that is generally valid also for other types of discretizations.

In the present work, we extend the constructive approach for LBM to \(d\) dimensions.
To the knowledge of the authors, this technique is the first top-down construction of an LBM for a given \(d\)-dimensional conservation law. 
The rest of the document is structured as follows. 
In Section~\ref{sec:meth} we introduce the target PDE, state the construction procedure, prove convergence of the relaxation system, assign specific stability parameters and, through discretization, obtain a second order LBM in space and time. 
Section~\ref{sec:numerics} discusses the numerical results and conclusions are drawn in Section~\ref{sec:conclusion}.

\section{Methodology}\label{sec:meth}

Within the construction procedure, we first transform the \(d\)-dimensional target PDE into an RS of size \((2d+1)\times (2d+1)\). 
For \(d=1\), the approach reduces to the previous one \cite{simonis2020relaxation}. 
Subsequently, we spectrally decompose the RS to obtain the transformed RS (TRS), which links to a discrete velocity Boltzmann model (DVBM). 
Discretizing the latter, we obtain a lattice Boltzmann equation (LBE) as a space-time evolution rule determining the final LBM.

\subsection{Target equation}
Let \(\rho: \Omega \times I \to \mathbb{R}, \left(\bm{x}, t\right) \mapsto \rho \left(\bm{x}, t\right) \) denote the conservative variable of the target equation (TEQ), which is an initial value problem (IVP) formed by a scalar, linear, \(d\)-dimensional advection\textendash diffusion equation (ADE)
\begin{align}\label{eq:targetADE}
\partial_{t} \rho + \bm{\nabla}_{\bm{x}} \cdot \bm{F} \left( \rho \right) - \mu \bm{\Delta}_{\bm{x}} \rho & = 0, \quad &&\text{in } \Omega \times I, \\
\rho\left( \cdot, 0\right) & \equiv \rho_{0} , \quad &&\text{in } \Omega,
\end{align}
where \(\bm{x}=\left(x, y, z\right)^{\mathrm{T}} \in \Omega \subseteq \mathbb{R}^{d}\), \( t \in I\subseteq \mathbb{R}_{0}^{+}\), \(\rho\) is periodic on \(\Omega\), \(\bm{F}\colon \mathbb{R} \to \mathbb{R}^{d}\) is linear, and \(\mu>0\) is a given diffusivity. 
Unless stated otherwise, we assume \(d=3\) and \(\bm{F} \left( \rho\right) \coloneqq \bm{u} \rho\) with a constant convection speed \(\bm{u}\in \mathbb{R}^{d}\).

\subsection{Constructing the relaxation system}
To approximate the TEQ \eqref{eq:targetADE} with LBM, we construct a generic RS via expanding the conservation law part of the PDE by perturbation terms \cite{simonis2020relaxation}. 
Let \(\tau_{\flat}\), \(a^{(1)}_{\alpha}\), \(a^{(2)}_{\alpha}\) define stability variables that need to be determined, where \(\alpha \in \{ 1, 2, \ldots, d\}\), \(\gamma > 0\), \(\delta = 2(\gamma -1)\), and \(\flat\) generalizes physical moment tensors \cite{simonis2021linear}. 
Unless stated otherwise, \(\partial_{\alpha} \coloneqq \partial / \partial x_{\alpha} \). 
Additionally, \(\cdot^{\epsilon}\) denotes a perturbed conservative variable, i.e. a quantity which solves the perturbed version of a PDE which initially is solved by \(\cdot\). 
The ansatz is based on the hyperbolic conservation law
\begin{align}\label{eq:conservationLaw}
\partial_{t} \rho + \bm{\nabla}_{\boldsymbol{x}} \cdot \boldsymbol{F} \left( \rho \right)   = 0, \quad &\text{in } \Omega \times I. 
\end{align}
To obtain an RS up to the first energy shell, two subsequent steps are performed. 
Each step consists of (i) introducing artificial variables (AV) and (ii) additional perturbation (AP) terms \cite{simonis2020relaxation}. 
In particular, for each \(\alpha\),
\begin{align}
& \underline{\text{AV}}: \phi_{\alpha} = F_{\alpha}\left(\rho\right) 
	&& \Rightarrow 
	\begin{cases} 
		\partial_{t} \rho + \sum\limits_{k=1}^{d}\partial_{k} \phi_{k} &= 0, \\
		0 &= F_{\alpha}\left(\rho\right)  - \phi_{\alpha}, 
	\end{cases} \\
& \underline{\text{AP}}: \epsilon^{\gamma} \tau_{\phi} \left( \partial_{t} \phi_{\alpha}^{\epsilon} + \frac{a_{\alpha}^{\left(1\right)}}{\epsilon^{\delta}}\partial_{\alpha} \rho^{\epsilon} \right)  = F_{\alpha}\left(\rho^{\epsilon}\right) - \phi_{\alpha}^{\epsilon} 
	&& \Rightarrow
	\begin{cases} 
		\partial_{t} \rho^{\epsilon} + \sum\limits_{k=1}^{d} \partial_{k} \phi_{k}^{\epsilon} &= 0, \\
		\partial_{t} \phi_{\alpha}^{\epsilon} + \frac{a_{\alpha}^{\left(1\right)}}{\epsilon^{\delta}}\partial_{\alpha} \rho^{\epsilon} &=  - \frac{1}{\epsilon^{\gamma} \tau_{\phi}} \left(  \phi_{\alpha}^{\epsilon} -  F_{\alpha}\left(\rho^{\epsilon} \right)  \right),
	\end{cases} \\
& \underline{\text{AV}}: \psi_{\alpha}^{\epsilon} = \frac{a_{\alpha}^{\left(1\right)}}{\epsilon^{\delta}} \rho^{\epsilon} 
	&& \Rightarrow
	\begin{cases}
		\partial_{t} \rho^{\epsilon} + \sum\limits_{k=1}^{d} \partial_{k} \phi_{k}^{\epsilon} &= 0, \\
		\partial_{t} \phi_{\alpha}^{\epsilon} + \partial_{\alpha} \psi_{\alpha}^{\epsilon} &=  - \frac{1}{\epsilon^{\gamma} \tau_{\phi}} \left( \phi_{\alpha}^{\epsilon} - F_{\alpha}\left(\rho^{\epsilon} \right)  \right), \\
		0 &=  \frac{a_{\alpha}^{\left(1\right)}}{\epsilon^{\delta}} \rho^{\epsilon} - \psi_{\alpha}^{\epsilon},
	\end{cases} \\
& \underline{\text{AP}}: \epsilon^{\gamma} \tau_{\psi} \left( \partial_{t} \psi_{\alpha}^{\epsilon\epsilon} + \frac{a_{\alpha}^{\left(2\right)}}{\epsilon^{\delta}}\partial_{\alpha} \phi_{\alpha}^{\epsilon\epsilon} \right)  = \frac{a_{\alpha}^{\left(1\right)}}{\epsilon^{\delta}} \rho^{\epsilon\epsilon} - \psi_{\alpha}^{\epsilon\epsilon} 
	&& \Rightarrow
	\begin{cases} 
		\partial_{t} \rho^{\epsilon\epsilon} + \sum\limits_{k=1}^{d} \partial_{k} \phi_{k}^{\epsilon\epsilon} &= 0, \\
		\partial_{t} \phi_{\alpha}^{\epsilon\epsilon} + \partial_{\alpha} \psi_{\alpha}^{\epsilon\epsilon}  &= - \frac{1}{\epsilon^{\gamma} \tau_{\phi}} \left(  \phi_{\alpha}^{\epsilon\epsilon} -  F_{\alpha}\left(\rho^{\epsilon\epsilon} \right)  \right), \\
		\partial_{t} \psi_{\alpha}^{\epsilon\epsilon} + \frac{a_{\alpha}^{\left(2\right)}}{\epsilon^{\delta}}\partial_{\alpha} \phi_{\alpha}^{\epsilon\epsilon} & = - \frac{1}{\epsilon^{\gamma} \tau_{\psi}} \left( \psi_{\alpha}^{\epsilon\epsilon} - \frac{a_{\alpha}^{\left(1\right)}}{\epsilon^{\delta}} \rho^{\epsilon\epsilon} \right), \label{eq:rsFull}
	\end{cases} 
\end{align}
in \(\Omega \times I\). 
Hence, we obtain a \((2d+1) \times (2d+1)\) system of equations \eqref{eq:rsFull} with relaxation terms on the right. 
\begin{definition} 
We write \eqref{eq:rsFull}, as a relaxation system 
\begin{align}\label{eq:rsClean}
\partial_{t} \boldsymbol{\rho}^{\epsilon\epsilon} + \sum\limits_{\alpha} \mathbf{A}_{\alpha} \partial_{\alpha} \boldsymbol{\rho}^{\epsilon\epsilon} = - \frac{1}{\epsilon^{\gamma}} \mathbf{S} \left[ \boldsymbol{\rho}^{\epsilon\epsilon} -  \hat{\boldsymbol{\rho}}^{\epsilon\epsilon} \right]  , 
\quad 
\mathbf{A}_{\alpha} = 
\begin{bmatrix}
0 & \boldsymbol{e}_{\alpha}^{\mathrm{T}} & \boldsymbol{0}_{1\times d} \\
\boldsymbol{0}_{d \times 1} & \boldsymbol{0}_{d\times d} & \mathrm{diag}\left(\boldsymbol{e}_{\alpha}^{\mathrm{T}}\right) \\
\boldsymbol{0}_{d\times 1} & \mathrm{diag} \left( \frac{a_{\alpha}^{(2)}}{\epsilon^{\delta}} \boldsymbol{e}_{\alpha}^{\mathrm{T}} \right) & \boldsymbol{0}_{d\times d} 
\end{bmatrix}
\in \mathbb{R}^{(2d+1) \times (2d+1)} , 
\end{align}
which governs the perturbed conservative variable \(\bm{\rho}^{\epsilon\epsilon} \coloneqq ( \rho^{\epsilon\epsilon}, 
(\boldsymbol{\phi}^{\epsilon\epsilon})^{\mathrm{T}} , (\boldsymbol{\psi}^{\epsilon\epsilon})^{\mathrm{T}} )^{\mathrm{T}} \in \mathbb{R}^{2d+1}\),
where \(\mathbf{A}_{\alpha}\) is diagonalizable by construction, \( \mathbf{S} = \mathrm{diag} ( \tau_{\rho}^{-1}, \tau_{\phi}^{-1} \mathbf{1}_{1\times d}, \tau_{\psi}^{-1} \mathbf{1}_{1\times d} ) \in \mathbb{R}^{(2d +1)\times (2d +1)} \) defines the relaxation matrix with \(\mathbf{r}_{a\times b} \in \mathbb{R}^{a\times b}\) being the all-\(r\) tensor of size \(a\times b\) for \(a,b \in \mathbb{N}\) and \(r\in \mathbb{R}\), and \(\boldsymbol{e}_{\alpha}\in\mathbb{R}^{d}\) denoting the \(\alpha\)th unit vector. 
The function \( \hat{\boldsymbol{\rho}}^{\epsilon\epsilon} \left( \rho^{\epsilon\epsilon} \right) \equiv ( 
\rho^{\epsilon\epsilon} ,
\boldsymbol{F}^{\mathrm{T}} ( \rho^{\epsilon\epsilon} ) ,
\frac{\rho^{\epsilon\epsilon}}{\epsilon^{\delta}} ( \boldsymbol{a}^{( 1 )})^{\mathrm{T}} )^{\mathrm{T}} \) is called equilibrium. 
\end{definition}
\begin{lemma}
Together with the initial condition given by \(
\boldsymbol{\rho}^{\epsilon\epsilon} \left(\cdot, 0 \right) \equiv \hat{\boldsymbol{\rho}}^{\epsilon\epsilon} \vert_{\rho^{\epsilon\epsilon}\left( \cdot , 0\right) = \rho_{0}}\) in \(\Omega\), \eqref{eq:rsClean} forms a well-posed IVP. 
\end{lemma}
\begin{proof}
Note that the equilibrium is solely dependent on \(\rho^{\epsilon\epsilon}\). 
The proof of a similar statement is given in \cite{bouchut2000diffusive}.
\end{proof}
\begin{theorem}
Let \(\bm{\rho}^{\epsilon\epsilon}\) be smooth in space and time. 
The RS \eqref{eq:rsFull} forms a closed equation for \(\rho^{\epsilon\epsilon}\), namely
\begin{align}
\partial_{t}\rho^{\epsilon\epsilon} 
+ \sum\limits_{\alpha} \partial_{\alpha} F_{\alpha} \left(\rho^{\epsilon\epsilon}\right) 
- \epsilon^{2-\gamma} \tau_{\phi} \sum\limits_{\alpha} a_{\alpha}^{(1)} \partial_{\alpha\alpha} \rho^{\epsilon\epsilon} 
= 
\epsilon^{\gamma}\tau_{\phi} 
\bigg[ &-
\left( 1 + \frac{\tau_{\psi}}{\tau_{\phi}}\right) \partial_{tt}\rho^{\epsilon\epsilon} 
- \frac{\tau_{\psi}}{\tau_{\phi}} \sum\limits_{\alpha} \partial_{\alpha t} F_{\alpha} \left( \rho^{\epsilon\epsilon} \right) \nonumber \\
& - \epsilon^{\gamma} \tau_{\psi} \partial_{ttt}\rho^{\epsilon\epsilon} 
+ \epsilon^{2-\gamma} \tau_{\psi} \sum\limits_{\alpha} a_{\alpha}^{(2)} \partial_{\alpha\alpha t} \rho^{\epsilon\epsilon} \bigg] . 
\end{align}
\end{theorem}
\begin{proof}
Recall the generalization of Schwarz's theorem for symmetric partial derivatives of arbitrary order and the fact that the graph of \(\bm{\phi}^{\epsilon\epsilon}\) approximates the linear flux \(\bm{F}\). 
Thus we perform an inverse recursive insertion from the last artificial variable to the initial conservation law \cite{simonis2020relaxation}. 
For any \(\alpha\), let the equations of the RS \eqref{eq:rsClean} be numbered as (I), (II)\(_\alpha\), and (III)\(_\alpha\), respectively.
Solve (III)\(_\alpha\) for \(\psi_{\alpha}^{\epsilon\epsilon}\) and \(\partial_{t}\)(II)\(_\alpha\) for \(\partial_{\alpha t} \psi_{\alpha}^{\epsilon\epsilon} = \partial_{t\alpha} \psi_{\alpha}^{\epsilon\epsilon}\), and insert both into (II)\(_{\alpha}\). 
The result in turn is solved for \(\phi_{\alpha}^{\epsilon\epsilon}\) and inserted into (I). 
Finally, computing \(\partial_{t}\)(I), \(\partial_{tt}\)(II), and \(\sum_{\alpha}\partial_{\alpha\alpha} a_{\alpha}^{(2)}\)(I) allows for substituting partial derivatives of \(\phi_{\alpha}^{\epsilon\epsilon}\) with expressions in \(\rho^{\epsilon\epsilon}\) and proves the claim. 
\end{proof}
\begin{remark}
The present ansatz enables both, constructing an RS and expressing the added relaxation terms as higher order derivatives, for any conservation law akin to \eqref{eq:conservationLaw} and thus any PDE which is transformable into a similar form.
\end{remark}

\subsection{Transformed relaxation system (TRS)}
\begin{definition}
With \(\chi_{\alpha}^{\left(i\right)} \coloneqq a_{\alpha}^{\left(i\right)}/\epsilon^{\delta}\) for \(i \in \{ 1,2\}\) and any \(\alpha\), define \(\mathbf{C}^{(i)} \coloneqq \mathrm{diag} ( \chi_{\alpha}^{(i)})_{\alpha} \in \mathrm{GL}_{d}(\mathbb{R})\) and 
\begin{align} \label{eq:diagMatrixD}
\mathbf{D} \coloneqq \begin{bmatrix}
\mathbf{1}_{1 \times d} 
	& 1 
	& \mathbf{1}_{1\times d} \\
- \left( \mathbf{C}^{(2)} \right)^{\circ \frac{1}{2}} 
	& \mathbf{0}_{d \times 1} 
	& \left( \mathbf{C}^{(2)}\right)^{\circ \frac{1}{2}}  \\
\mathbf{C}^{(2)}
	& \mathbf{0}_{d \times 1}
	& \mathbf{C}^{(2)}
\end{bmatrix} \in \mathrm{GL}_{2d + 1}\left(\mathbb{R}\right) ,
 \quad
\mathbf{D}^{-1} = 
\begin{bmatrix}
\boldsymbol{0}_{d \times 1 } & - \frac{1}{2} \left( \mathbf{C}^{(2)} \right)^{\circ - \frac{1}{2}}  & \frac{1}{2} \left( \mathbf{C}^{(2)} \right)^{\circ - 1} \\
1 & \boldsymbol{0}_{1\times d} & - \left( \frac{1}{\chi_{\alpha}^{(2)}} \right)_{\alpha} \\
\boldsymbol{0}_{d \times 1} & \frac{1}{2} \left( \mathbf{C}^{(2)} \right)^{\circ - \frac{1}{2}}  & \frac{1}{2} \left( \mathbf{C}^{(2)} \right)^{\circ - 1}
\end{bmatrix} , 
\end{align}
where the \(\circ\)-exponents denote Hadamard operations \cite{reams1999hadamard}. 
For any \(\alpha\), \(\mathbf{A}_{\alpha}\) can be diagonalized with \( \mathbf{A}_{\alpha}^{\mathrm{d}} \coloneqq \mathbf{D}^{-1} \mathbf{A}_{\alpha} \mathbf{D} 
= \mathrm{diag} (
( \mathbf{C}^{(2)} )^{\circ( 1 / 2) } \bm{e}_{\alpha} , \mathbf{0}_{1\times 1} , ( \mathbf{C}^{(2)} )^{\circ (1/2)} \bm{e}_{\alpha} ) \). 
\end{definition}
\begin{definition}
Spectrally decomposing the RS \eqref{eq:rsClean}, we define the vector \(\bm{g} \coloneqq \mathbf{D}^{-1} \boldsymbol{\rho}^{\epsilon\epsilon}\) which is governed by the TRS
\begin{align}\label{eq:trs}
\partial_{t} \boldsymbol{g} + \sum\limits_{\alpha} \mathbf{A}^{\mathrm{d}}_{\alpha} \partial_{\alpha}\boldsymbol{g} = - \frac{1}{\epsilon^{\gamma}} \mathbf{D}^{-1} \mathbf{S} \mathbf{D}\left[ \boldsymbol{g} - \boldsymbol{G}\left( \boldsymbol{g} \right) \right] .
\end{align}
Here, \(\boldsymbol{G} \coloneqq \bm{\mathcal{G}} \circ \iota \circ \mathcal{D}\) such that \(\boldsymbol{G}\left(\boldsymbol{g}\right) \overset{!}{=} \mathbf{D}^{-1} \hat{\boldsymbol{\rho}}^{\epsilon\epsilon} \), 
where \(\iota:\boldsymbol{\rho}^{\epsilon\epsilon} \mapsto \rho^{\epsilon\epsilon}\) extracts the non-artificial variables, the linear map induced by \(\mathbf{D}\) is \(\mathcal{D}\colon \mathbb{R}^{2d+1} \to \mathbb{R}^{2d+1}\), and, with \(\bm{a} \coloneqq ( a_{\alpha}^{(1)} / a_{\alpha}^{(2)} )_{\alpha} \in \mathbb{R}^{d}\), the generalized Maxwellian is defined as 
\begin{align}\label{equ:generalizedMaxw}
\bm{\mathcal{G}}: 
\left[0, 1\right] \times \mathbb{R} 
\to 
\mathbb{R}^{2d+1}, 
\left(\epsilon, \eta \right) 
\mapsto \bm{\mathcal{G}} \left( \epsilon, \eta \right) = \left( \mathcal{G}_{1}, \ldots, \mathcal{G}_{2d+1}\right)^{\mathrm{T}}\left( \epsilon, \eta \right) = 
\begin{pmatrix} 
\frac{1}{2} \left[ \bm{a} \eta - \left( \mathbf{C}^{(2)}\right)^{\circ - \frac{1}{2}} \bm{F} \left( \eta\right) \right]\\
\left(1 - \mathbf{1}_{d\times 1} \cdot \bm{a}  \right) \eta  \\
\frac{1}{2} \left[ \bm{a} \eta + \left( \mathbf{C}^{(2)}\right)^{\circ - \frac{1}{2}} \bm{F} \left( \eta\right) \right] \end{pmatrix}. 
\end{align}
\end{definition}
The preceding derivation enables the algebraic characterization of the collision, the AV space, as well as the equilibrium which completely determine the relaxation limit of the generic RS and in turn the relaxation procedure of LBM. 
\begin{remark}[Collision]
The multi-relaxation-time (MRT) collision matrix \(\mathbf{K} \coloneqq  \mathbf{D}^{-1} \mathbf{S} \mathbf{D}\) is explicitely computed as 
\begin{align}\label{eq:matrix-K}
 \mathbf{K} = 
\begin{bmatrix}
\frac{1}{2}\left( \frac{1}{\tau_{\phi}} + \frac{1}{\tau_{\psi}} \right) \mathbf{I}_{d} & \boldsymbol{0}_{d \times 1} & - \frac{1}{2}\left( \frac{1}{\tau_{\phi}} - \frac{1}{\tau_{\psi}} \right) \mathbf{I}_{d} \\
\left( \frac{1}{\tau_{\rho}} - \frac{1}{\tau_{\psi}}\right) \boldsymbol{1}_{1 \times d} & \frac{1}{\tau_{\rho}} & \left( \frac{1}{\tau_{\rho}} - \frac{1}{\tau_{\psi}}\right) \boldsymbol{1}_{1\times d}\\
- \frac{1}{2}\left(  \frac{1}{\tau_{\phi}} - \frac{1}{\tau_{\psi}} \right) \mathbf{I}_{d} & \boldsymbol{0}_{d \times 1} &  \frac{1}{2}\left(\frac{1}{\tau_{\phi}} + \frac{1}{\tau_{\psi}} \right) \mathbf{I}_{d} 
\end{bmatrix} .
\end{align}
In comparison to a single-relaxation-time (SRT) collision \(\breve{\mathbf{K}}\coloneqq \mathbf{D}^{-1} \breve{\mathbf{S}} \mathbf{D} = \tau_{\rho}^{-1} \mathbf{I}_{2d+1}\), where \(\breve{\mathbf{S}} \coloneqq \tau_{\rho}^{-1} \mathbf{I}_{2d+1}\), the MRT collision carries off-diagonal entries which correlate non-equilibrium contributions via relaxation frequency sums. 
\end{remark}
\begin{remark}[AV space]
The choice of the appearance of AP and AV (\(\boldsymbol{\phi}^{\epsilon\epsilon}\) and \(\boldsymbol{\psi}^{\epsilon\epsilon}\)) within the constructive ansatz, already determines the unified diagonalizer \(\mathbf{D}\) which in turn defines the structure of the TRS via \(\mathbf{A}^{\mathrm{d}}_{\alpha}\) and \(\boldsymbol{G}\). 
In particular, 
\begin{align}\label{eq:colspD}
\mathrm{colsp} \left( \mathbf{D} \right) = \mathrm{span} \left( \bigcap\limits_{\alpha} E \left( \mathbf{A}_{\alpha} \right) \right) 
\end{align}
determines the possibilities for \(\mathbf{D}\), where \(E\left( \mathbf{A}_{\alpha}\right) \) denotes the eigenbasis of \(\mathbf{A}_{\alpha}\) consisting of right eigenvectors. 
In response of both, the choice of the AV and that the TRS is obtained through eigendecomposition of \(\mathbf{A}_{\alpha}\), we limit our discussion on orthogonal moment bases. 
\end{remark}
\subsection{Convergence result}
Let \(\Xi = \{ \eta \in \mathbb{R} : \vert \eta \vert \leq \Vert \rho_{0} \Vert_{\infty} \}\) and \(\boldsymbol{F}(0) = \mathbf{0}_{d\times 1}\).
Thus \(\forall \epsilon \in (0,1]\) it holds \(\bm{\mathcal{G}}(\epsilon, 0 ) \equiv \mathbf{0}_{(2d+1)\times 1}\) and we assume that \(\mathcal{G}_{i} (\epsilon, \cdot)\) is non-decreasing in \(\Xi\) respectively for all \(i\in \{1, 2, \ldots, 2d+1\}\). 
In \cite{simonis2020relaxation} the stability structures \cite{rheinlander2010stability} are proven to coincide with the sub-characteristics condition \cite{jin1995relaxation,bouchut2000diffusive}. 
Hence, we proceed with evaluating the latter. 
\begin{lemma}
The generalized Maxwellian \(\bm{\mathcal{G}}\) admits conditions \((M_{1}\)\textendash \(M_{4})\) in \cite{bouchut2000diffusive}. 
\end{lemma}
\begin{proof}
Unless stated otherwise, let \(i=1, \ldots, 2d+1\). 
Some algebra verifies 
\begin{align}
& ( M_{1} ) \quad 
\sum\limits_{i} \mathcal{G}_{i}\left(\epsilon, \eta \right) = \eta \quad 
&& \forall \epsilon \in \left(0, 1\right] ~\forall \eta \in \Xi,  \label{equ:bouchutProofM1}\\
& ( M_{2} ) \quad 
\sum\limits_{i} \left( \mathbf{A}^{\mathrm{d}}_{\alpha} \right)_{i,i} \mathcal{G}_{i}\left(\epsilon, \eta \right) = F_{\alpha}\left( \eta \right) 
&& \forall \alpha \in \{1,2,\ldots, d\} ~\forall \epsilon \in \left(0, 1\right]~ \forall \eta \in \Xi,   \\
& ( M_{3} ) \quad 
\sum\limits_{i} \left[ \sqrt{\epsilon^{\delta}} \left( \mathbf{A}^{\mathrm{d}}_{\alpha} \right)_{i,i} \right] \Bigl[ \sqrt{\epsilon^{\delta}} \bigl( \mathbf{A}^{\mathrm{d}}_{\beta} \bigr)_{i,i} \Bigr]  \mathcal{G}_{i}\left(0, \eta\right) = \mu \eta \delta_{\alpha,\beta}
&&\forall \alpha, \beta \in \{1,2,\ldots, d\} ~\forall \eta \in \Xi,   \\
& ( M_{4} ) \quad 
\lim\limits_{\epsilon \searrow 0} \mathcal{G}_{i}\left(\epsilon, \eta\right) = \mathcal{G}_{i} \left( 0, \eta \right)  
&& \text{uniformly for } \eta \in \Xi, 
\end{align}
where \((M_{3})\) requires that \(\forall \alpha \colon a^{(1)}_{\alpha} = \mu \). 
\end{proof}

\begin{definition}
The TRS is termed relaxation-stable if the stability constants are chosen such that \(\forall \alpha\colon  a^{(1)}_{\alpha} = \mu\) and 
\begin{align}\label{equ:stability_nondecreasing}
a^{(2)}_{\alpha} \geq a^{(1)}_{\alpha} 
\quad \wedge \quad 
\frac{a^{(1)}_{\alpha}}{\sqrt{\epsilon^{\delta} a^{(2)}_{\alpha}}} \geq \vert F^{\prime}_{\alpha} \left( \rho \right)\vert. 
\end{align}
\end{definition}

\begin{lemma} \label{prop:convergence_RS}
Let \(\tau_{\flat} = 1\) \(\forall \flat \in \{\rho, \phi, \psi\}\) and \(\rho_{0} \in L^{\infty}\left( \Omega \right) \cap L^{1}\left( \Omega \right) \). 
Initialize the TRS \eqref{eq:trs} with \(\bm{g}\left(\cdot, 0\right) = \bm{\mathcal{G}}\left(\epsilon, \rho_{0}\right) \) and specify the stability constants such that the TRS is relaxation-stable. 
Then 
\begin{align}
\lim\limits_{\epsilon \searrow 0 }  \rho^{\epsilon\epsilon} =  \rho_{\star} \in C ( I; L_{\mathrm{loc}}^{1} \left( \Omega \right) ) \cap L^{\infty}\left( \Omega  \times I \right)
\end{align}
is the unique solution to the TEQ \eqref{eq:targetADE}. 
\end{lemma}
\begin{proof}
The claim follows from \cite[Corollary 3.1]{simonis2020relaxation}. 
\end{proof}

\begin{remark}
With Lemma~\ref{prop:convergence_RS}, we have also verified structural stability of the TRS, in case of \(\tau_{\flat} = 1\). 
For LBM, the conditions \eqref{equ:stability_nondecreasing} represent a positivity-preserving bound and a linear stability criteria of the equilibrium distribution, respectively. 
The LBM-counterpart of \eqref{equ:stability_nondecreasing} is sufficient for stability under the premise of uniform relaxation \cite{hosseini2017stability}. 
\end{remark}

\subsection{Discrete velocity Boltzmann model}
We adapt the notation of \cite{simonis2020relaxation}. 
A description of the \(D3Q7\) discrete velocity stencil is given for example in \cite{siodlaczek2021numerical}.
\begin{theorem}
Let \(\tau = 1\), \(\gamma = 2\), and \(\mathbf{S}_{k} \coloneqq \mathrm{diag} ( \mathsf{s}_{k})\) with the component maps \(\mathsf{s}_{k} \coloneqq  ( (\bm{c}_{j} )_{k} )_{0\leq j < q}^{\mathrm{T}}\).
The \(D3Q7\) DVBE 
\begin{align} \label{equ:3x3dvbe}
\partial_{t} 
\bm{\mathsf{f}} + 
\sum\limits_{k=1}^{d} \mathbf{S}_{k} 
\partial_{k} 
\bm{\mathsf{f}} 
= 
- \frac{1}{\epsilon^{\gamma}\tau} 
\left( \bm{\mathsf{f}} - \bm{\mathsf{f}}^{\mathrm{eq}} \left( \bm{\mathsf{f}} \right) \right) ,
\end{align}
with initial condition \(\bm{\mathsf{f}}\left(\bm{x}, 0\right) = \bm{\mathsf{f}}^{\mathrm{eq}} \left( \bm{\mathsf{f}} \right) \coloneqq \varepsilon ( \rho_{0}; \bm{\mathsf{s}} )\) converges in \(C ( I; L_{loc}^{1} \left(\Omega \right) ) \) to the unique solution of the IVP \eqref{eq:targetADE}. 
\end{theorem}
\begin{proof}
We link the TRS \eqref{eq:trs} to the \(D3Q7\) DVBE via assigning the stability parameters. 
Given the constants \(\lambda, \theta > 0 \), set \(\tau_{\flat} \coloneqq \tau\), \(a_{\alpha}^{(1)} \coloneqq \lambda^{2}/\theta\), and \(a_{\alpha}^{(2)} \coloneqq \lambda^{2}\) \(\forall \alpha\). 
Thus, with \(\varepsilon \left( \cdot ; \bm{\mathsf{s}}\right) \equiv \bm{\mathcal{G}} \left(\epsilon, \cdot\right)\) and \(\tau = 1\), Lemma~\ref{prop:convergence_RS} implies convergence.
\end{proof}

\subsection{Lattice Boltzmann method}
Using the limit consistent second order discretization of \eqref{equ:3x3dvbe} with a Crank\textendash Nicolson-type method \cite{simonis2022limit}, we obtain an LBE evolving the populations \(f_{i}\) in space-time on \(D3Q7\) with SRT collision 
\begin{align} 
f_{i}\left( \bm{x} + \triangle t \bm{c}_{i}, t + \triangle t \right) = f_{i} \left( \bm{x},  t \right) - \frac{\triangle t}{\tau - \frac{\triangle t}{2} } \left[f_{i} \left( \bm{x}, t \right) - f_{i}^{\mathrm{eq}} \left( \bm{x}, t \right) \right], \quad \text{for } i=0, 1 \ldots, 6 . 
\end{align}
Following \cite{simonis2022limit}, the embedded limit yields convergence to the solution of the target IVP \eqref{eq:targetADE} up to a truncation error of \(\mathcal{O} \left( \triangle x^{2} \right)\) in diffusive scaling \(\gamma = 2 \) and \(\epsilon \mapsfrom \triangle t \sim \triangle x^{2} \), where \(\mu = c_{s}^{2} ( \tau - \triangle t / 2 )\). 

\section{Numerical tests}\label{sec:numerics}
All computations were done with OpenLB release 1.5 \cite{kummerlander2022olb15} on at most four nodes with two Intel Xeon Platinum 8368 CPUs and four NVIDIA A100-40 GPUs each. 
The experimental order of convergence (EOC) is evaluated with specific choices of \(\Omega\), \(I\), \(\bm{u}\), \(\mu\), and \(\rho_0\) for benchmark tests from \cite{dapelo2021lattice-boltzmann}.
Let \(\Omega = \left(-1, 1\right)^{3}\ni \left( x, y, z \right)^{\mathrm{T}} = \bm{x}\) and \(I = (t_0, t_M)\).
We use SI units with characteristic scales \(l_{\mathrm{c}} = 2 [\mathrm{m}]\) and \(u_{\mathrm{c}} = 2.5 [\mathrm{m}/\mathrm{s}]\) and neglect further notation. 
A relative \(L^{2}\)-error with respect to the analytical solution \(\rho_{\star}\) is averaged in \((t_{0}, t_{M}) = (0, 1.52)\) to measure an overall error 
\(\overline{\mathrm{err}} = ( 1 / M ) \sum_{i=1}^{M} ( \{ \sum_{\bm{x}\in\Omega} [ \rho ( \bm{x}, t_{i}) - \rho_{\star} ( \bm{x}, t_{i} ) ]^{2} \}/ \{ \sum_{\bm{x}\in \Omega}  [ \rho_{\star} ( \bm{x}, t_{i} ) ]^{2} \} )^{1/2}\). 
We compute samples \((N, P\!e) \in \mathfrak{N} \times \mathfrak{P}\), where \(\mathfrak{N} = \{2^{n} \times 25 : n\in \{0,1,\ldots, 5\}\}\) and \(\mathfrak{P} = \{10^{n} : n\in \{2, 3, 4, 5 \}\}\), and thus test a range of grid P{\'e}clet \(P\!e_{\mathrm{g}} = P\!e / N \) and Courant numbers \(C\!o = (u_{\mathrm{c}} \triangle t)/ \triangle x\). 
The results of the computations in \(\mathfrak{N} \times \mathfrak{P}\) under diffusive scaling for the following examples are compiled in Figure~\ref{fig:smoothADEerr}.
\paragraph{Example 1: Smooth initial data}
The IVP \eqref{eq:targetADE} with \(\rho_{0}^{(\mathrm{sm})}\left(\bm{x}\right) = \sin \left( \pi x\right) \sin \left( \pi y\right) \sin \left( \pi z\right)+1\) is analytically solved by  
\begin{align}
\rho_{\star}^{(\mathrm{sm})}\left(\bm{x}, t\right) =  \mathrm{sin} \left( \pi\left[ x - u_{x} t \right]\right) 
			\mathrm{sin} \left( \pi\left[ y - u_{y} t \right]\right) 
			\mathrm{sin} \left( \pi\left[ z - u_{z} t \right]\right) 
			\mathrm{exp} \left( - 3 \mu \pi^{2}t \right) + 1.
\end{align}
\paragraph{Example 2: Non-smooth initial data}
We initialize \eqref{eq:targetADE} with a superposition of Gaussian hills along the \(x\)-axis.  
To realize a non-differentiability which persists only for \(P\!e_{g} \nearrow \infty\), we set an initial in-domain peak
\begin{align}
\rho_{0}^{(\mathrm{ns})} \left(\bm{x} \right) = \begin{cases} 
\frac{1}{\sqrt{4\pi \mu \triangle t}} + 1, & \quad \text{if } x \in \left(-\frac{\triangle x }{ 2} , \frac{\triangle x}{2} \right), \\
1 , & \quad \text{otherwise}. 
\end{cases} 
\end{align}
The analytical solution is formed through diffusion transport of a Dirac comb 
\begin{align}\label{eq:nonSmoothNonPer}
\rho^{(\mathrm{ns})}_{\star} \left(\bm{x}, t \right) = 
\frac{1}{\sqrt{4\pi \mu t}} \sum\limits_{k\in\mathbb{Z}} \mathrm{exp}\left( - \frac{\left[x - \left(x_{0} + u_{x} t\right) + 2k\right]^{2}}{4 \mu t} \right) + 1, 
\quad
\lim_{t\searrow 0} \rho^{(\mathrm{ns})}_{\star} \left(\bm{x}, t\right) 
 =  \Sha_{2}\left(x-x_{0}\right) + 1 , 
\end{align}
where \(x_{0}\) denotes the \(x\)-location of the peak at \(t=0\). 
Further details and a proof of the limit \(t\searrow 0\) are given in \cite{dapelo2021lattice-boltzmann}. 

{\color{black}
\begin{figure}[ht!]
	\centerline{
		\subfloat[\(P\!e=10^{2}\)]{			
		    \includegraphics[scale=1]{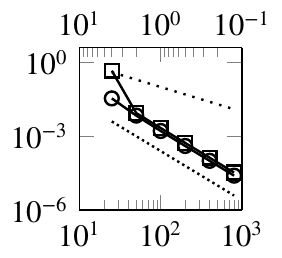}
			\label{sfig:errorSmoothPe100}
		}
		\subfloat[\(P\!e=10^{3}\)]{
		    \includegraphics[scale=1]{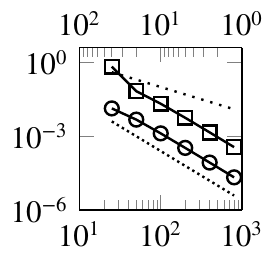}
			\label{sfig:errorSmoothPe1000}
		}
		\subfloat[\(P\!e=10^{4}\)]{
		    \includegraphics[scale=1]{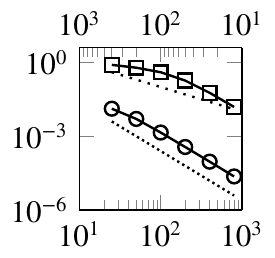}
			\label{sfig:errorSmoothPe10000}
		}
		\subfloat[\(P\!e=10^{5}\)]{
		    \includegraphics[scale=1]{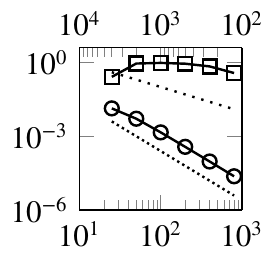}
			\label{sfig:errorSmoothPe100000}
		}
		\subfloat{
		    \includegraphics[scale=1]{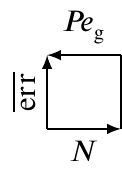}
			\hspace{0.5em}
		    \includegraphics[scale=1]{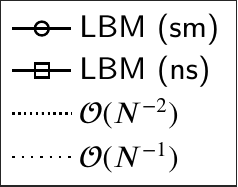}
		}
	}
	\caption{Errors of \(D3Q7\) SRT LBM approximating \eqref{eq:targetADE} with smooth (sm) and non-smooth (ns) initial data.}
	\label{fig:smoothADEerr}
\end{figure}
}
The spatio-temporal grid size is \(\triangle t = \triangle x^{2} =  ( l_{\mathrm{c}} / N )^{2}\) with convection speeds of \(\bm{u}^{(\mathrm{sm})} = u_{\mathrm{c}} \mathbf{1}_{d\times 1}\) and \(\bm{u}^{(\mathrm{ns})} = u_{\mathrm{c}} \bm{e}_{x}\) and the corresponding relaxation time \(\tau\). 
The Courant number sequences over \(C\!o = 0.5^{n} \times 0.2\) with \(n \in \{0,1,\ldots,4\}\), while several magnitudes of grid P{\'e}clet numbers, \(P\!e_{\mathrm{g}} \in [\mathcal{O} (10^{-1} ), \mathcal{O} (10^{3} ) ]\) are swept. 
Whereas Figure~\ref{fig:smoothADEerr} approves the EOC of two for the smooth IVP, a reduction from second to first order is clearly visible for non-smooth initialization at \(P\!e_{\mathrm{g}} \gtrsim 10^{2}\).
Increasing \(P\!e_{\mathrm{g}} \nearrow 10^{4}\) in the latter case, induces a larger error contribution breaking also the EOC of one, which agrees to previous results with \(D3Q19\) \cite{dapelo2021lattice-boltzmann}. 
Since, for further increase of \(P\!e_{\mathrm{g}}\), the non-smooth initialization exits the function space in Lemma~\ref{prop:convergence_RS}, a blowup is expected due to delayed smoothing.

\section{Conclusion}\label{sec:conclusion}
A novel procedure to construct an RS for a given \(d\)-dimensional ADE is established. 
Subsequently, the RS is linked to a \(DdQ(2d+1)\) DVBM on the zeroth and first energy shell. 
With that, we extend the top-down design of LBM \cite{simonis2020relaxation} to \(d\) dimensions. 
Additionally, the necessary LBM ingredients represented by the moment space, the collision scheme, and the equilibrium, are algebraically characterized at the relaxation level. 
A closed equation with general scaling for the RS unfolds the approximation order of the relaxation terms. 
The DVBM is proven to converge to the solution of the target IVP.
The second order discretization of the DVBM leads to an LBM of spatio-temporal order two.  
We provide numerical tests of the top-down constructed LBM for smooth and non-smooth initial data in \(d=3\) dimensions via computing over several ranges of grid-normalized non-dimensional numbers. 
The numerical results indicate that the second order convergence in space for smooth initial data reduces to first order and eventually breaks, when sharpening the initial peak towards a non-smooth delta function. 
Future studies should include solutions to this observation via dynamic MRT stabilization \cite{simonis2021linear,simonis2022temporal}, or entropy control \cite{karlin2014gibbs} of artificial relaxation parameters. 

\section*{Acknowledgement}
This work was performed on the HoreKa supercomputer funded by the Ministry of Science, Research and the Arts Baden-W{\"u}rttemberg and by the Federal Ministry of Education and Research.

\printcredits


\end{document}